\documentstyle[12pt]{article} \vsize=28.7truecm \hsize=23truecm \columnsep=0.8truecm

 \def\vbar{\mathchoice{\vrule height2.3ptdepth-.3ptwidth.12pt\kern-
 .10pt}
    {\vrule height6.3ptdepth-.3ptwidth.11pt\kern-.11pt}
    {\vrule height5.1ptdepth-.30ptwidth.8pt\kern-.8pt}
    {\vrule height4.1ptdepth-.24ptwidth.6pt\kern-.7pt}}
\def\reel{\hbox{I\hskip-2pt R}}
\def\comp{\hbox{I\hskip-6pt C}}

\def\<{\langle}
\def\>{\rangle}

\def\n{{\boldmath n}}

\def\<{\langle}
\def\>{\rangle}
\def\mathben{\hbox{I\hskip -2pt N}}

\def\mathbb{\hbox{I\hskip -2pt 1}}
\def\reel{\hbox{I\hskip -2pt R}}

\def\n{{\noindent}}
=0 in \headheight =0 cm

\newtheorem{theorem}{Theorem}[section]
\newtheorem{definition}{Definition}[section]

\newtheorem{prop}{Proposition}[section]

\newtheorem{remark}{Remark}[section]

\newtheorem{proposition}{Proposition}[section]

\newenvironment{proof}{\hspace*{0cm}{\bf Proof}}{\hfill $\Box$
\vspace*{0.5cm}}

\begin{document}

 \vspace{0.5cm}
\title{{\bf
Beta-hypergeometric probability distribution on symmetric matrices}}
\author{ A. Hassairi \footnote{Corresponding author.
 \textit{E-mail address: abdelhamid.hassairi@fss.rnu.tn}} , M. Masmoudi, O. Regaig
\\{\footnotesize{\it
Sfax University Tunisia.}}}

 \date{}
 \maketitle{Running title: \emph{Beta-hypergeometric distribution }}

\n $\overline{\hspace{15cm}}$\vskip0.3cm  \n {\small {\bf Abstract}}
{\small  : Some remarkable properties of the beta distribution are
based on relations involving independence between beta random
variables such that a parameter of one among them is the sum of the
parameters of an other (see (1.1) et (1.2) below). Asci, Letac and
Piccioni \cite{6} have used the real beta-hypergeometric
distribution on $ \reel$ to give a general version of these
properties without the condition on the parameters. In the present
paper, we extend the properties of the real beta to the beta
distribution on symmetric matrices, we use on the positive definite
matrices the division algorithm defined by the Cholesky
decomposition to define a matrix-variate beta-hypergeometric
distribution, and we extend to this distribution the proprieties
established in the real case by Asci, Letac and Piccioni.
}\\

\n {\small {\it{ Keywords:}} Hypergeometric function,
Beta-hypergeometric distribution, symmetric matrices, generalized
power,
spherical Fourier transform.\\
$\overline{\hspace{15cm}}$\vskip1cm

\section{Introduction}

\hspace{0.4cm}Consider the gamma distribution on $ \reel, $ with
scale parameter $\sigma>0$ and shape parameter $p>0,$
$$\gamma_{p,\sigma}(dy)=\frac{\sigma^{p}}{\Gamma(p)}e^{-\sigma y }y^{p-1}{\bf{1}}_{(0,+\infty)}(y)dy.$$
Let  $ U$ and $ V$ be two independent random variables with
respective gamma distributions $\gamma_{p,\sigma},$
$\gamma_{q,\sigma}, $ and define $$ X=\frac{U}{U+V} \ \textrm{and}\
Y=\frac{U}{V}.$$ Then  the distribution of $ X$ and $ Y$ are called
the beta distributions of the first and of the second kind with
parameters $(p, q)$ and are denoted by $ \beta^{(1)}_{p, q}$
and $\beta^{(2)}_{p, q}$ respectively.\\
The beta distributions of the first and of second kind on $\reel$
have many remarkable properties. For instance, it is well known (see
\cite{6}) that
\begin{equation}\label{Ab1}
\textrm{if}\ \  W' \sim \beta^{(2)}_{a+a', a'}\ \   \textrm{is
independent of}\ \ X\sim \beta^{(1)}_{a, a'},\ \ \textrm{then}\ \
\frac{1}{1+W'X} \sim \beta^{(1)}_{a', a}.
\end{equation}
And it is shown in \cite{7} that
\begin{equation}\label{Ab2}
\textrm{if}\ \ W \sim \beta^{(2)}_{a+a', a}, \ X\sim \beta^{(1)}_{a,
a'}\  , \ \ W'\sim \beta^{(2)}_{a+a', a'}\ \textrm{are independent,
then} \ \ \frac{1}{1+\frac{W}{1+W'X}}\sim X.
\end{equation}
In these two properties, the random variables $W$ and $W'$ are beta
distributed with first parameter equal to the sum of the parameters
of the distribution of the variable $X$. Asci, Letac and Piccioni
\cite{6} have used the so-called real beta-hypergeometric
distribution to extend these results to the general case where $
W\sim \beta^{(2)}_{b, a}, \ W'\sim \beta^{(2)}_{b, a'} $ with $ b
>0$ not necessarily equal to $ a+a'$. Recall that the hypergeometric
function $_pF_q$ is defined for positive numbers
$a_{1},...,a_{p};b_{1},...,b_{q},$ by
$$_pF_q(a_{1},...,a_{p};b_{1},...,b_{q};x)=\sum_{n=0}^{\infty}\frac{(a_{1})_{n}...(a_{p})_{n}}{n!
(b_{1})_{n}...(b_{q})_{n}}x^{n}, \textrm{with} \
(a)_{n}=\frac{\Gamma(a+n)}{\Gamma(a)}.
$$ The beta-hypergeometric distribution with parameters $ (a, a', b) $
is then defined by $$ \mu_{a,a',b}(dx)= C(a,a',b)x^{a-1}(1-x)^{b-1}
\hspace{2mm}_2F_1(a,b;a+a';x) {\mathbf{1}}_{(0,1)}(x) dx,
$$ where $$ C(a,a',b)=\frac{\Gamma(a+b)}{\Gamma(a) \Gamma(b) _3F_2(a,a,b;a+b,a+a';1)}.$$
 Note that the distribution $
\mu_{a,a',b},$ reduces to a $\beta^{(1)}_{a,a'}$ when
$ b=a+a'.$\\
Asci, Letac and Piccioni have shown that
\begin{equation}\label{Ab3}
\textrm{if}\ \ X \sim \mu_{a,a',b}\  \textrm{and} \ W'\sim
\beta^{(2)}_{b, a'}\  \textrm{are independent, then} \
\frac{1}{1+W'X} \sim \mu_{a',a,b}
\end{equation}
and
\begin{equation}\label{Ab4}
\textrm{if}\ \  W \sim \beta^{(2)}_{b, a},\ W'\sim \beta^{(2)}_{b,
a'}\  \textrm{and} \  X>0  \ \textrm{are  independent,}\
\textrm{then}\  \end{equation} $$ X\sim
\frac{1}{1+\frac{W}{1+W'X}} \  \textrm{if and only if } \  X \sim
\mu_{a,a',b}.$$

In the present work, we first extend the properties in (\ref{Ab1})
and (\ref{Ab2}) to the beta distributions on symmetric matrices. We
then use results from harmonic analysis on symmetric cones, and a
division algorithm defined by the Cholesky decomposition to extend
the definition of a beta-hypergeometric distribution to the cone of
positive definite symmetric matrices generalizing the definition of
a beta distribution on matrices (see \cite{9}). These distributions
are then used to extend to symmetric matrices the results
established in the real case by Asci, Letac and Piccioni. It is
worth mentioning here that in a private communication, Letac has
given a definition of a beta-hypergeometric distribution on
symmetric matrices using a division algorithm based on the notion of
quadratic representation. The use of the division algorithm defined
by the Cholesky decomposition is crucial in our work, it allows the
calculation of the spherical Fourier transform which the key tool in
some proofs. The paper has the following plan. In Section 2, after a
review of some facts concerning the beta distributions on symmetric
matrices, we establish the matrix versions of the properties
(\ref{Ab1}) and (\ref{Ab2}). We then introduce the matrix
beta-hypergeometric distribution and we show some preliminary
properties concerning this distribution. In Section 3, we state and
prove our main results, in particular we use the matrix
beta-hypergeometric distribution to give a matrix version of
(\ref{Ab3}) and (\ref{Ab4}).

 \section{Matrix variate beta-hypergeometric distribution}
\hspace{0.4cm}
Let $V$ be the linear space of symmetric $r\times r$ matrices on
$\reel $, and $\Omega$ be the cone of positive definite elements
of $V$. We denote the identity matrix by $e$, the determinant of
an element $x$ of $V$ by $\Delta(x)$ and its trace by $
\textrm{tr} x.$ We equip $V$ with the inner product $\langle x,y
\rangle=\textrm{tr}(xy)$, for all $x,y \in V$. For an invertible
$r\times r$ matrix $a$, we consider the automorphism $g(a)$ of $V$
defined by $g(a)x=axa^{*}$ where $a^{*} $ is the transpose of $ a
.$ We denote $G$ the group of such isomorphisms, $K$ the subgroup
of elements of $G$ corresponding to $a$ orthogonal, called the
orthogonal group. As mentioned above, we will use the division
algorithm on matrices based on the Cholesky decomposition of an
element $y$ of $\Omega$, that is on the fact that $y$ can be
written in a unique manner as $ y=tt^{*}$, where $t$ is a lower
triangular matrix with strictly positive diagonal (see
\cite{cas}). For an element $x$ in V, we set $\pi(y)(x)=txt^{*},$
and we define the "quotient"
of $x$ by $y$ as $ \pi^{-1}(y)(x)=t^{-1}x(t^{*})^{-1},$ which, for simplicity, we sometimes denote abusively $\frac{x}{y}.$\\
 Consider the absolutely continuous Wishart distribution concentrated on $\Omega$
 with shape parameter $  p>(r-1)/2$ and scale parameter $\sigma\in\Omega,$
$$
 W_{p,\sigma}(dx)=\frac{(\Delta(\sigma))^{-p}}{\Gamma_{\Omega}(p)}\exp(- \textrm{tr}
(x
\sigma^{-1}))(\Delta(x))^{p-\frac{r+1}{2}}{\mathbf{1}}_{\Omega}(x)dx,
$$
where
 $$  \Gamma_{\Omega}(p)=(2\pi)^{\frac{r(r-1)}{4}}\prod_{k=1}^{r}\Gamma(p-(k-1)/2).$$
If $ U$ and $ V$ are two independent Wishart random matrices with
the same scale parameter $ \sigma$ and respective shape parameters $
p>\frac{r-1}{2} \ \textrm {and}  \ q>\frac{r-1}{2} , $ then the
random matrix $ \pi^{-1}(U+V)(U)$ has the so called beta
distribution (of the first kind) $\beta^{(1)}_{p,q}$ on $\Omega \cap
(e-\Omega)$ given by
$$ \beta^{(1)}_{p,
q}(dx)=(B_{\Omega}(p, q))^{-1} (\Delta(x))^{p-\frac{r+1}{2}}
(\Delta(e-x))^{q-\frac{r+1}{2}} {\mathbf{1}}_{\Omega \cap
(e-\Omega)}(x) dx,$$
 where the normalizing constant $ B_{\Omega}(p,
q)$ is the multivariate beta function defined by $$ \ B_{\Omega}(p,
q)=\frac{\Gamma_{\Omega}(p) \Gamma_{\Omega}(q)
}{\Gamma_{\Omega}(p+q)}.$$ We also have the beta distribution (of
the second kind) $\beta^{(2)}_{p, q} $ on $ \Omega$ given by
$$\beta^{(2)}_{p, q}(dx)=(B_{\Omega}(p, q))^{-1}
(\Delta(x))^{p-\frac{r+1}{2}} (\Delta(e+x))^{-(p+q)}
{\mathbf{1}}_{\Omega }(x) dx.
$$ It is the distribution of the random matrix $ \pi^{-1}(V)(U)$, or equivalently the distribution of $ \pi^{-1}(e-Z)(Z)$ with $ Z \sim \beta^{(1)}_{p,
q}.$ More precisely, we have:
\begin{prop}\label{2.1}
Let $Y$ be a random matrix in $\Omega$. Then $Y\sim
\beta^{(2)}_{p,q}$ if and only if $
Z=\pi^{-1}(e+Y)(Y)\sim\beta^{(1)}_{p,q}.$
\end{prop}
\begin{proof}
Let $Y\sim \beta^{(2)}_{p, q}$ and $Z= \pi^{-1}(e+Y)(Y)$ which is
equivalent to $ Y=\pi^{-1}(e-Z)(Z).$ \\For a bounded measurable
function $h,$ we have
\begin{eqnarray*}
E(h(Z))&=& (B_{\Omega}(p, q))^{-1}\int_{\Omega}
h(\pi^{-1}(e+y)(y)) \Delta(y)^{p-\frac{r+1}{2}}
\Delta(e+y)^{-(p+q)} dy\\
&=& (B_{\Omega}(p, q))^{-1}\int_{\Omega \cap (e-\Omega)} h(z)
 \Delta(
\pi^{-1}(e-z)(z))^{p-\frac{r+1}{2}} \Delta(\pi^{-1}(e-z)(e))^{-(p+q)} \Delta(e-z)^{-(r+1)} dz \\
&=& (B_{\Omega}(p, q))^{-1}\int_{\Omega \cap (e-\Omega)}
h(z)\Delta(z)^{p-\frac{r+1}{2}} \Delta(e-z)^{q-\frac{r+1}{2}} dz.
\end{eqnarray*}
Thus $ Z\sim \beta^{(1)}_{p, q}.$\\
In a same way, we verify that if $ Z\sim \beta^{(1)}_{p, q} $ then $
\pi^{-1}(e-Z)(Z)\sim \beta^{(2)}_{p, q}.$
\end{proof}

Now, we give the matrix versions of $ (\ref{Ab1})$ and $
(\ref{Ab2}).$
\begin{theorem}
Let $ W', \ X$ and $ W$ be three independent random matrices valued in $ \Omega.$\\
\begin{equation} \label{Ab5} i) \  \textrm{If}\ \  W' \sim \beta^{(2)}_{a+a', a'} \
\  \textrm{and} \ \  X \sim \beta^{(1)}_{a, a'},\ \ then\
\end{equation} $$ \pi^{-1}(e+\pi(X)(W'))(e)\sim \beta^{(1)}_{a',
a}.$$
\begin{equation}\label{Ab6}
 ii) \ \textrm{If}  \ \   W \sim \beta^{(2)}_{a+a', a},\  X \sim
\beta^{(1)}_{a, a'} \ \  \textrm{and} \ \     W' \sim
\beta^{(2)}_{a+a', a'} \ \   ,then
\end{equation}
$$\pi^{-1}(e+\pi^{-1}(e+\pi(X)(W'))(W))(e)\sim X.$$
\end{theorem}
\begin{proof} i) Let $W'$ and $ X$ be two independent random matrix such that $ W'
\sim \beta^{(2)}_{a+a', a'} $ and $ X \sim \beta^{(1)}_{a, a'}.$ It
is known (see Theorem 2.2 in \cite{BF}), that $ \pi(X)(W') \sim
\beta^{(2)}_{a, a'} $, and according to Proposition \ref{2.1}, we
obtain that $\pi^{-1}(e+\pi(X)(W'))(\pi(X)(W'))\sim \beta^{(1)}_{a,
a'}.$ It follows that $$\pi^{-1}(e+\pi(X)(W'))(e)=
e-[\pi^{-1}(e+\pi(X)(W'))(\pi(X)(W'))] \sim \beta^{(1)}_{a', a}.$$

ii) As $\pi^{-1}(e+\pi(X)(W'))(W)=
\pi(\pi^{-1}(e+\pi(X)(W')(e))(W),$ then according to Theorem 2.2 in
\cite{BF} and to $(\ref{Ab5})$, we obtain that
$\pi^{-1}(e+\pi(X)(W'))(W) \sim \beta^{(2)}_{a', a}.$
\\Therefore $\pi^{-1}(e+\pi^{-1}(e+\pi(X)(W'))(W))(e)\sim X.$
\end{proof}

In what follows, we will be interested in the extension of the
results in Theorem 2.1 to the case where the parameter $a+a'$ in the
distributions of $W$ and $W'$ is replaced by any parameter $b$ not
necessary equal to $a+a'$. For this we require some further terminology.\\
Let $\mathcal{P}$ denote the space of polynomials on the space $V$
of symmetric $r\times r$ matrices. A natural representation $
\mathcal{H}$ of the group of automorphisms of $V$ is defined for $g$
in this group and $p$ in $\mathcal{P}$ by $
({\mathcal{H}}(g)p)(x)=p(g^{-1} x)$.
\\For $X=(X_{ij})_{1\leq i,j\leq r}\ \textrm{ in}\ \Omega \
\textrm{and }\ 1\leq k\leq r $, let $\Delta_k(X) $ denote the
principal minor of order $k$ of $X,$ that is the determinant of
the sub-matrix $ P_{k}(X)=((X_{ij})_{1\leq i, j \leq k }).$ The
generalized power of $X$ is defined for $s=(s_{1},...,s_{r})\in
\comp^{r} $, by
$$ \Delta_{s}(X)=(\Delta_{1}(X))^{s_{1}-s_{2}}
(\Delta_{2}(X))^{s_{2}-s_{3}} ... (\Delta_{r}(X))^{s_{r}} .$$ For a
given $m=(m_1,\ldots,m_r)\in \mathben^{r}$ which satisfies $m_1\geq
m_2\geq\ldots\geq m_r\geq 0$ (denoted by $ m\geq 0$), we denote by
${\mathcal{P}}_{m}$ the subspace of $\mathcal{P}$ generated by the
polynomials ${\mathcal{H}}(g)\Delta_m$ where $g\in G$. The spherical
polynomial $\phi_m$ is defined in \cite{FK} by
$$\phi_m(x)=\int_K\Delta_m(kx)dk.$$
Up to a constant factor, the $\phi_m$ are the only $K$-invariant polynomials in ${\mathcal{P}}_{m}$.\\
The definition of the beta-hypergeometric distribution on the cone
of positive definite symmetric matrices relies on the notion of
hypergeometric function which appears in  \cite{FK}, page $318$. For
instance, for $a=(a_1,\ldots,a_r)$ in $\comp^r$ and
$m=(m_1,\ldots,m_r)$ in $\mathben^{r}$ such that $ m\geq 0$, we
write
$$(a)_m=\frac{\Gamma_\Omega(a+m)}{\Gamma_\Omega(
a)}=\prod_{i=1}^r({a}_{i}-\frac{i-1}{2})_{m_i},
$$
and for $\alpha_i=(\alpha_i^1,\ldots,\alpha_i^r)$ in $\comp^r$,
$i=1,\ldots,p$ and $\beta_j=(\beta_j^1,\ldots,\beta_j^r)$ in
$\comp^r$, $j=1,\ldots,q,$ we define the hypergeometric function
\begin{equation}\label{R21}
_pF_q(\alpha_1,\ldots,\alpha_p;\beta_1,\ldots,\beta_q;x)=\sum_{m\geq0}\frac{(\alpha_1)_m\dots
(\alpha_p)_m}{(\beta_1)_m\dots(\beta_q)_m}
\frac{1}{(\frac{n}{r})_m}d_m \phi_m(x), \end{equation} where $d_m$
is the dimension of ${\mathcal{P}}_{m}.$ Note that we can define the hypergeometric function for some complex $\alpha_i$ and $\beta_j$, $i=1,\ldots,p$,
 $j=1,\ldots,q$, where$
 (\alpha_{i})_m=\prod_{j=1}^r({\alpha}_{i}-\frac{j-1}{2})_{m_j}$
 and $
 (\beta_{j})_m=\prod_{i=1}^r({\beta}_{j}-\frac{i-1}{2})_{m_i}.
 $
 \\It is shown in \cite{FK},
page 318, that the domain $\mathcal{D}$ of convergence of this
series is:\begin{itemize}
                                \item $V$, if $p\leq q$,
                                \item $D=\{w; |w|<1\}$ where $|.|$ is the spectral norm, if $p=q+1$,
                                \item $\emptyset$, if $p>q+1$.
                              \end{itemize}
As it is done for the real beta-hypergeometric distribution
(see\cite{6}), we will be interested in the case $p=q+1$. We will
first show that under some conditions, the definition of
$_{q+1}F_q(x)$ may be extended to $x=e$. This is in fact equivalent
to show that the series (\ref{R21}) converges when $x=e$.
\begin{prop}\label{2.2}
 Let
$\alpha_j=(\alpha_j^1,\ldots,\alpha_j^r)$ and
$\beta_j=(\beta_j^1,\ldots,\beta_j^r)$ in $\reel^r$, such that
$\alpha_j^i \neq \frac{i-1}{2}   $ and $\beta_j^i \neq
\frac{i-1}{2} $ for all $i=1,\ldots,r.$ Denote for $i=1,\ldots,r,$
$$ c_i=\sum_{1\leq j\leq q}\beta_j^i - \sum_{1\leq j\leq
q+1}\alpha_j^i. $$
 Then the
series
\begin{equation}\label{**}
\sum_{m\geq0}\frac{(\alpha_1)_m\dots
(\alpha_{q+1})_m}{(\beta_1)_m\dots(\beta_q)_m}
\frac{1}{(\frac{n}{r})_m}d_m,
\end{equation}
 converges if and only if, for all $1\leq k\leq r$,
 $$\sum_{i=1}^k c_i> 1+k(r-k)-\frac{k(r+1)}{2}.$$
 \end{prop}
 Note that in particular $c_1>\frac{r-1}{2}$, and
 when $r=1$, the condition reduces to $c_1>0$.\\
\begin{proof}
Let $\alpha_j=(\alpha_j^1,\ldots,\alpha_j^r) \in \reel^r$ and
$\beta_j=(\beta_j^1,\ldots,\beta_j^r)$ in $\reel^r.$ We will
consider separately two cases:
\begin{itemize}
\item Case where $ \alpha_{j}^{i}
> \frac{i-1}{2}$ and $ \beta_{j}^{i} > \frac{i-1}{2}$ for all $i=1,\ldots,r.$
\\ Denote $p_i=m_i-m_{i+1}$, for $i=1,\ldots,r$, where
$m_{r+1}=0.$ Then $(p_1,\ldots p_r)\in \mathben^r$ and $ m_i=\sum_{k=i}^r p_k .$  \\
Using the fact that
$$ d_m\simeq\prod_{1\leq i<j\leq r}(1+m_i-m_j),$$ (see \cite{FK},
page 286), we obtain that
$$ d_m\simeq\prod_{1\leq i<j\leq r}(1+\sum_{k=i}^{j-1}p_k).$$ On the other hand, by Stirling
approximation, we have that for $ m_{i} \neq 0,$
$$(\alpha_j^i-\frac{i-1}{2})_{m_i}\sim
\frac{m_i^{(\alpha_j^i-\frac{i+1}{2})}m_i!}{\Gamma(\alpha_j^i-\frac{i-1}{2})}.$$
Then $$(\alpha_j)_m\sim
\prod_{i=1}^r\frac{m_i^{(\alpha_j^i-\frac{i+1}{2})}m_i!}{\Gamma(\alpha_j^i-\frac{i-1}{2})},
\  \ j=1,\ldots,q+1.$$
 Consequently, the term of the series (\ref{**}) is equivalent to
$$ A_1A_2\ldots  A_r\prod_{i=1}^r(\sum_{k=i}^r p_k)^{-c_i-\frac{n}{r}}\prod_{1\leq i<j\leq r}(1+\sum_{k=i}^{j-1}p_k),$$
 where
$A_i=\frac{\Gamma(\frac{n}{r}-\frac{i-1}{2})\prod_{k=1}^q\Gamma(\beta_k^i-\frac{i-1}{2})}{\prod_{k=1}^{q+1}
\Gamma(\alpha_k^i-\frac{i-1}{2})} $ and $c_i=\sum_{1\leq j\leq
q}\beta_j^i-\sum_{1\leq j\leq q+1}\alpha_j^i$, $i=1,\ldots,r$.\\
Hence, the series (\ref{**}) converges if and only if  $\textrm{for
all} \ \  1\leq k\leq r,$
$$\sum_{i=1}^k c_i> 1+k(r-k)-\frac{k(r+1)}{2}\ \ .$$
\item Case where $  \alpha_{j}^{i} < \frac{i-1}{2}$ or $
\beta_{j}^{i} < \frac{i-1}{2}$ for some $ i=1,\ldots,r.$ There
exists $ k \in \mathben $ such that $ -k < \alpha_{j}^{i}-
\frac{i-1}{2} <-k+1.$ This implies that $ (\alpha_{j}^{i}-
\frac{i-1}{2})_{m_{i}}= (\alpha_{j}^{i}- \frac{i-1}{2})_{k}
(\alpha_{j}^{i}- \frac{i-1}{2}+k)_{m_{i}-k}$ for $ m_{i}\geq k.$
Then also the series (\ref{**}) is convergent if and only if
$\sum_{i=1}^k c_i> 1+k(r-k)-\frac{k(r+1)}{2}\ \  \textrm{for all}\ \
1\leq k\leq r.$
\end{itemize}
\end{proof}\\
Note that if $  \alpha_{j}^{i} =\frac{i-1}{2}$, for some
$i=1,...,r,$ then in the case where $ m_{i}=0,$ \ $
(\alpha_{j})_{m}=\prod_{k=1, k\neq i}^{r}(\alpha_{j}^{k}
-\frac{k-1}{2}).$ If not $(\alpha_{j})_{m}= 0.$\\ Hence the series
(\ref{**}) is convergent  if and only if
$$\sum_{j=1}^k c_j> 1+k(r-k)-\frac{k(r+1)}{2}\ \  \textrm{for all}
\ \ 1\leq k\leq i-1.$$

We are now in position to introduce the beta-hypergeometric distribution
on symmetric matrices.
\begin{definition} The beta-hypergeometric distribution, with parameters $ (a, a', b) \in (]\frac{r-1}{2},+\infty[)^3$,
 is defined on $ \Omega \cap (e-\Omega)$ by \begin{equation}\label{d1}
\mu_{a,a',b}(dx)=C(a,a',b)\Delta(x)^{a-\frac{n}{r}}\Delta(e-x)^{b-\frac{n}{r}}\hspace{2mm}_2F_1(a,b;a+a';x)
\textbf{1}_{\Omega \cap (e-\Omega)}(x)(dx),
\end{equation}
where
$$C(a,a',b)=\frac{\Gamma_\Omega(a+b)}{\Gamma_\Omega(a)\Gamma_\Omega(b)\hspace{2mm}_3F_2(a,a,b;a+b,a+a';e)}.$$
\end{definition}
Note that the distribution $\mu_{a,a',a+a'}$ is nothing but the
distribution $\beta^{(1)}_{a,a'}$. In fact, since we
have
\begin{equation}\label{3.14}
_2F_1(a,b;a+a';x)=\Delta(e-x)^{a'-b}\hspace{2mm}_2F_1(a',a+a'-b;a+a';x),
\end{equation}
(see  \cite{FK}, page $330$), then (\ref{d1}) becomes
\begin{equation}\label{3.15}
\mu_{a,a',b}(dx)=\frac{\Gamma_\Omega(a+b)\Gamma_\Omega(a')}{\Gamma_\Omega(a+a')\Gamma_\Omega(b)\hspace{2mm}_3F_2(a,a,b;a+b,a+a';e)}\hspace{2mm}_2F_1(a',a+a'-b;a+a';x)
\beta^{(1)}_{a,a'}(dx).
\end{equation}
When $a+a'-b=0$, $\hspace{2mm}_3F_2(a,a,b;a+b,a+a';e)$
becomes$\hspace{2mm}_2F_1(a,a;2a+a';e).$
This, using the following Gauss formula
\begin{equation}\label{3.100}
_2F_1(\alpha,\beta;\gamma;e)=\frac{\Gamma_\Omega(\gamma)
\Gamma_\Omega(\gamma-\alpha-\beta)}{\Gamma_\Omega(\gamma-\beta)\Gamma_\Omega(\gamma-\alpha)},
\end{equation}
for $ \alpha=\beta=a$ and $ \gamma=2a+a'$ shows that
$\mu_{a,a',a+a'}$ coincides with
$\beta^{(1)}_{a,a'}.$ \\
Next, we calculate the spherical Fourier transform of a
beta-hypergeometric distribution, it is the expectation of its
generalized power. This transform is important, it plays, for the
$K$-invariant distributions on symmetric matrices, the role that the
Mellin transform plays in the real case.
\begin{prop}
Let $X$ be a random variable having the beta-hypergeometric
distribution $\mu_{a,a',b}$ defined by (\ref{d1}). Then for $
t=(t_1,\ldots,t_r) \in \reel^{r}$ such that $t_i+a>\frac{i-1}{2},$
for all $ 1\leq i\leq r, $ the spherical Fourier transform of $X$ is
\begin{equation}\label{delta}
E(\Delta_t(X))= \frac{\Gamma_\Omega(a+b)}{\Gamma_\Omega(a)}
\frac{\Gamma_\Omega(t+a)}{\Gamma_\Omega(t+a+b)}\frac{_3F_2(a,b,a+t;a+a',t+a+b;e)}{_3F_2(a,a,b;a+b,a+a';e)}.
\end{equation}
\end{prop}
\begin{proof}
\begin{eqnarray*}
 E(\Delta_t(X))&=&C(a,a',b)\int_{\Omega\cap
(e-\Omega)}\Delta_t(x)\Delta(x)^{a-\frac{n}{r}}\Delta(e-x)^{b-\frac{n}{r}}\hspace{2mm}_2F_1(a,b;a+a';x)dx.
\end{eqnarray*}
Since the determinant and the hypergeometric function are
$K$-invariant, for $k \in K,$
\begin{eqnarray*}
E(\Delta_t(X))&=&C(a,a',b)\int_{\Omega\cap
(e-\Omega)}\Delta_t(x)\Delta(k^{-1}x)^{a-\frac{n}{r}}\Delta(e-k^{-1}x)^{b-\frac{n}{r}}\hspace{1mm}_2F_1(a,b;a+a';
k^{-1}x)dx.
\end{eqnarray*}
\\
With the change of variable $y=k^{-1}x$, we can write
\begin{eqnarray*}
E(\Delta_t(X))&=&C(a,a',b)\int_{\Omega\cap
(e-\Omega)}\Delta_t(ky)\Delta(y)^{a-\frac{n}{r}}\Delta(e-y)^{b-\frac{n}{r}}\hspace{1mm}_2F_1(a,b;a+a';y)dy
\\&=&C(a,a',b)\int_{\Omega\cap
(e-\Omega)}\Delta_t(ky)\Delta(ky)^{a-\frac{n}{r}}\Delta(e-ky)^{b-\frac{n}{r}}\hspace{1mm}_2F_1(a,b;a+a';y)dy
\\&=&C(a,a',b)\sum_{m\geq0}\frac{(a)_m(b)_m d_m}{(a+a')_m(\frac{n}{r})_m}\int_{\Omega\cap
(e-\Omega)}\int_K\Delta_t(ky)\Delta(ky)^{a-\frac{n}{r}}\Delta(e-ky)^{b-\frac{n}{r}}\Delta_m(ky)dk
dy\\
&=&C(a,a',b)\sum_{m\geq0}\frac{(a)_m(b)_m}{(a+a')_m(\frac{n}{r})_m}d_m\int_{\Omega\cap
(e-\Omega)}\Delta_{m+t+a-\frac{n}{r}}(y)\Delta(e-y)^{b-\frac{n}{r}}dy
\\&=&\sum_{m\geq0}\frac{(a)_m(b)_m}{(a+a')_m(\frac{n}{r})_m}d_m\frac{\Gamma_\Omega(a+b)}{\Gamma_\Omega(a)\Gamma_\Omega(b)\hspace{2mm}
_3F_2(a,a,b;a+b,a+a';e)}\frac{\Gamma_\Omega(m+t+a)\Gamma_\Omega(b)}{\Gamma_\Omega(m+t+a+b)}
\\&=&\sum_{m\geq0}\frac{(a)_m(b)_m(t+a)_m}{(a+a')_m(t+a+b)_m(\frac{n}{r})_m}d_m\frac{\Gamma_\Omega(a+b)}{\Gamma_\Omega(a)\hspace{2mm}
_3F_2(a,a,b;a+b,a+a';e)}\frac{\Gamma_\Omega(t+a)}{\Gamma_\Omega(t+a+b)}
\\&=&\frac{\Gamma_\Omega(a+b)}{\Gamma_\Omega(a)}
\frac{\Gamma_\Omega(t+a)}{\Gamma_\Omega(t+a+b)}\frac{_3F_2(a,b,a+t;a+a',t+a+b;e)}{_3F_2(a,a,b;a+b,a+a';e)}.
\end{eqnarray*}
\end{proof}

\section{Characterizations of the beta-hypergeometric distributions}
In this section, we state and prove our main results involving
the beta-hypergeometric probability measure $\mu_{a,a',b}$.
\begin{theorem}\label{th1}
Let $X$ and $ W$ be two independent random
matrices such that $W\sim \beta^{(2)}_{b,a'}$ and $X\sim
\mu_{a,a',b}$. Then
\begin{equation}\label{2.7}
\pi^{-1}(e+ \pi (X)(W))(e)\sim \mu_{a',a,b}.
\end{equation}
\end{theorem}
For the proof, we need to establish the following technical result.
\begin{prop}\label{0.1}
For $a$, $a'$, $b$ $ \in ]\frac{r-1}{2}, \infty[$ and $z\in
\Omega\cap (e-\Omega)$, we have
\begin{equation}
\int_{\Omega\cap
(e-\Omega)}\frac{\Delta(e-t)^{a+a'-\frac{n}{r}}\Delta(t)^{b-\frac{n}{r}}}{\Delta(e-\pi(z)(t))^{a'+b}}\hspace{1mm}_2F_1(a,b;a+a';e-t)dt
=\frac{\Gamma_\Omega(a')\Gamma_\Omega(b)}{\Gamma_\Omega(a'+b)}\hspace{1mm}_2F_1(a',b;a+a';z).
\end{equation}
\end{prop}
\begin{proof}
We again use the invariance by the orthogonal group $K$ of the
determinant and of the hypergeometric function. For $k\in K$, we
have
\begin{eqnarray*}
I&=&\int_{\Omega\cap (e-\Omega)}
\frac{\Delta(e-t)^{a+a'-\frac{n}{r}}\Delta(t)^{b-\frac{n}{r}}}{\Delta(e-\pi(z)(t))^{a'+b}}
\hspace{1mm}_2F_1(a,b;a+a';e-t)dt\\ &=& \int_{\Omega\cap
(e-\Omega)}\frac{\Delta(e-k^{-1}t)^{a+a'-\frac{n}{r}}\Delta(k^{-1}t)^{b-\frac{n}{r}}}{\Delta(e-\pi(z)(t))^{a'+b}}
\hspace{1mm}_2F_1(a,b;a+a';e-k^{-1}t)dt.
\end{eqnarray*}
Setting $y=k^{-1}t$ in last integral, we get
\begin{eqnarray*}
I&=&\int_{\Omega\cap
(e-\Omega)}\frac{\Delta(e-y)^{a+a'-\frac{n}{r}}\Delta(y)^{b-\frac{n}{r}}}
{\Delta(e-\pi(z)(ky))^{a'+b}}\hspace{1mm}_2F_1(a,b;a+a';e-y)dy\\
&=&\int_{\Omega\cap
(e-\Omega)}\frac{\Delta(e-ky)^{a+a'-\frac{n}{r}}\Delta(ky)^{b-\frac{n}{r}}}
{\Delta(e-\pi(z)(ky))^{a'+b}}\hspace{1mm}_2F_1(a,b;a+a';e-y)dy\\
&=&\sum_{m\geq0}\frac{(a)_m(b)_m}{(a+a')_m(\frac{n}{r})_m}d_m\int_{\Omega\cap
(e-\Omega)}\int_K
\frac{\Delta(e-ky)^{a+a'-\frac{n}{r}}\Delta(ky)^{b-\frac{n}{r}}}
{\Delta(e-\pi(z)(ky))^{a'+b}} \Delta_m(e-ky) dk dy \\
&=&\sum_{m\geq0}\frac{(a)_m(b)_m}{(a+a')_m(\frac{n}{r})_m}d_m\int_{\Omega\cap
(e-\Omega)}
\frac{\Delta(e-y)^{a+a'-\frac{n}{r}}\Delta(y)^{b-\frac{n}{r}}}
{\Delta(e-\pi(z)(y))^{a'+b}}\Delta_m(e-y) dy
\\
&=&\sum_{m\geq0}\frac{(a)_m(b)_m}{(a+a')_m(\frac{n}{r})_m}d_m\int_{\Omega\cap
(e-\Omega)}
\frac{\Delta_{m+a+a'-\frac{n}{r}}(e-y)\Delta(y)^{b-\frac{n}{r}}}
{\Delta(e-\pi(z)(y))^{a'+b}}dy
\\&=&\sum_{m\geq0}\frac{(a)_m(b)_m}{(a+a')_m(\frac{n}{r})_m}d_m\frac{\Gamma_\Omega(b)\Gamma_\Omega(m+a+a')}
{\Gamma_\Omega(m+a+a'+b)}\hspace{1mm}_2F_1(a'+b,b;b+m+a+a';z)
\\&=&\sum_{m\geq0}\frac{(a)_m(b)_m}{(a+a')_m(\frac{n}{r})_m}d_m\frac{\Gamma_\Omega(b)\Gamma_\Omega(m+a+a')}
{\Gamma_\Omega(m+a+a'+b)}\sum_{k\geq0}\frac{(a'+b)_k(b)_k}{(b+m+a+a')_k(\frac{n}{r})_k}d_k\phi_k(z)
\\&=&\sum_{k\geq0}\frac{(a'+b)_k(b)_k}{(\frac{n}{r})_k}d_k\phi_k(z)
\sum_{m\geq0}\frac{(a)_m(b)_m}{(a+a')_m(\frac{n}{r})_m}d_m\frac{\Gamma_\Omega(b)\Gamma_\Omega(m+a+a')}
{\Gamma_\Omega(m+a+a'+b)(b+m+a+a')_k}
\\&=&\sum_{k\geq0}\frac{(a'+b)_k(b)_k}{(\frac{n}{r})_k}d_k\phi_k(z)
\sum_{m\geq0}\frac{(a)_m(b)_md_m}{(a+a')_m(\frac{n}{r})_m}\frac{\Gamma_\Omega(b)\Gamma_\Omega(m+a+a')}
{\Gamma_\Omega(m+a+a'+b)}\frac{\Gamma_\Omega(b+a+a'+m)}{\Gamma_\Omega(b+m+a+a'+k)}
\\&=&\sum_{k\geq0}\frac{(a'+b)_k(b)_k}{(\frac{n}{r})_k}d_k\phi_k(z)\sum_{m\geq0}\frac{(a)_m(b)_md_m}{(b+a+a'+k)_m
(\frac{n}{r})_m}\frac{\Gamma_\Omega(b)\Gamma_\Omega(a+a')}
{\Gamma_\Omega(b+a+a'+k)}
\\&=&\sum_{k\geq0}\frac{(a'+b)_k(b)_k}{(\frac{n}{r})_k}d_k\phi_k(z)\frac{\Gamma_\Omega(b)\Gamma_\Omega(a+a')}
{\Gamma_\Omega(b+a+a'+k)}\hspace{2mm}_2F_1(a,b;a+a'+b+k;e)
\\&=&\sum_{k\geq0}\frac{(a'+b)_k(b)_k}{(\frac{n}{r})_k}d_k\phi_k(z)\frac{\Gamma_\Omega(b)\Gamma_\Omega(a+a')}
{\Gamma_\Omega(b+a+a'+k)}\frac{\Gamma_\Omega(a+a'+b+k)\Gamma_\Omega(a'+k)}{\Gamma_\Omega(a+a'+k)\Gamma_\Omega(a'+b+k)}
\\&=&\sum_{k\geq0}\frac{(b)_k(a')_k}{(a+a')_k(\frac{n}{r})_k}d_k\phi_k(z)\frac{\Gamma_\Omega(b)\Gamma_\Omega(a')}{\Gamma_\Omega(a'+b)}
\\&=&\frac{\Gamma_\Omega(b)\Gamma_\Omega(a')}{\Gamma_\Omega(a'+b)}\hspace{2mm}_2F_1(a',b;a+a';z).
\end{eqnarray*}
\end{proof}

We come now to the proof of Theorem 3.1.\\
\begin{proof}\textbf{ of Theorem \ref{th1}}
 Let $X'$ be a random variable with distribution $
\mu_{a',a,b}$, and define $V=\pi^{-1}(X')(e-X').$ Then showing
(\ref{2.7}) is equivalent to show that $V $ and $\pi(X)(W)$ have the
same distribution. Let $h$ be a bounded measurable function. Then
\begin{eqnarray*}
E(h(V))&=&\int_{\Omega\cap
(e-\Omega)}h(\pi^{-1}(x)(e-x))\mu_{a',a,b}(dx)
\\&=&C(a',a,b)\int_{\Omega\cap
(e-\Omega)}h(\pi^{-1}(x)(e-x))\Delta(x)^{a'-\frac{n}{r}}\Delta(e-x)^{b-\frac{n}{r}}\hspace{2mm}_2F_1(a',b;a+a';x)dx.
\end{eqnarray*}
Setting $y=\pi^{-1}(x)(e-x)$, or equivalently $x=\pi^{-1}(e+y)(e)$,
then $dx =\Delta(e+y)^{-\frac{2n}{r}}dy$, and we have
\begin{eqnarray*}
E(h(V))&=& C(a',a,b)\int_\Omega
h(y)\Delta(\pi^{-1}(e+y)(e))^{a'-\frac{n}{r}}\Delta(e-\pi^{-1}(e+y)(e))^{b-\frac{n}{r}}\\
& \ \ &\hspace{2mm}_2F_1(a',b;a+a';\pi^{-1}(e+y)(e))\Delta(e+y)^{-\frac{2n}{r}}dy\\
 &=& C(a',a,b)\int_\Omega
h(y)\Delta(\pi^{-1}(e+y)(e))^{a'+b}\Delta(y)^{b-\frac{n}{r}}\hspace{2mm}_2F_1(a',b;a+a';\pi^{-1}(e+y)(e))dy.
\end{eqnarray*}
Hence the density of $V$ is
\begin{equation}\label{2..10}
f_V(v)=C(a',a,b)\Delta(\pi^{-1}(e+v)(e))^{a'+b}\Delta(v)^{b-\frac{n}{r}}\hspace{2mm}_2F_1(a',b;a+a';\pi^{-1}(e+v)(e))
\textbf{1}_\Omega(v).
\end{equation}
On the other hand, the density of $U=\pi(X)(W)$ is given by
\begin{eqnarray*}
f_U(u)&=&\int_{\Omega\cap (e-\Omega)}
f_X(x)f_{W}(\pi^{-1}(x)(u))\Delta(x)^{-\frac{n}{r}}dx
\\&=&C(a,a',b)\frac{\Gamma_\Omega(a'+b)}{\Gamma_\Omega(a')\Gamma_\Omega(b)}\int_{\Omega\cap (e-\Omega)}
\Delta(x)^{a-\frac{n}{r}}\Delta(e-x)^{b-\frac{n\mathbf{}}{r}}\hspace{2mm}_2F_1(a,b;a+a';x)
\Delta(\pi^{-1}(x)(u))^{b-\frac{n}{r}}
\end{eqnarray*}
$\hspace{1.5cm}\Delta(e+\pi^{-1}(x)(u))^{-b-a'}\Delta(x)^{-\frac{n}{r}}dx$
\begin{eqnarray*}
\hspace{1cm}&=&C(a,a',b)\frac{\Gamma_\Omega(a'+b)}{\Gamma_\Omega(a')\Gamma_\Omega(b)}\Delta(u)^{b-\frac{n}{r}}\int_{\Omega\cap
(e-\Omega)}\Delta(x)^{a+a'-\frac{n}{r}}\Delta(e-x)^{b-\frac{n}{r}}\Delta(x+u)^{-b-a'}
 \end{eqnarray*}
 $\hspace{1.5cm}_2F_1(a,b;a+a';x) dx.$\\
With the change $t=e-x$, we get
\begin{eqnarray*}
f_U(u)&=&C(a,a',b)\frac{\Gamma_\Omega(a'+b)}{\Gamma_\Omega(a')\Gamma_\Omega(b)}\Delta(u)^{b-\frac{n}{r}}\int_{\Omega\cap
(e-\Omega)}\Delta(e-t)^{a+a'-\frac{n}{r}}\Delta(t)^{b-\frac{n}{r}}\Delta(e+u-t)^{-b-a'}
\end{eqnarray*}
$\hspace{1.5cm}_2F_1(a,b;a+a';e-t)dt \\$
\begin{eqnarray*}
&=&C(a,a',b)\frac{\Gamma_\Omega(a'+b)}{\Gamma_\Omega(a')\Gamma_\Omega(b)}\frac{\Delta(u)^{b-\frac{n}{r}}}{\Delta(e+u)^{a'+b}}
\int_{\Omega\cap (e-\Omega)}
\frac{\Delta(e-t)^{a+a'-\frac{n}{r}}\Delta(t)^{b-\frac{n}{r}}}
{\Delta(e-\pi^{-1}(e+u)(t))^{b+a'}}
\end{eqnarray*}
$\hspace{1.5cm} _2F_1(a,b;a+a';e-t)dt.$\\
Using the fact that $ \pi^{-1}(e+u)(t)=\pi(\pi^{-1}(e+u)(e))(t)$,
and invoking Proposition 3.1, we obtain that\\
 $\int_{\Omega\cap
 (e-\Omega)}\frac{\Delta(e-t)^{a+a'-\frac{n}{r}}\Delta(t)^{b-\frac{n}{r}}}
{\Delta(e-\pi(\pi^{-1}(e+u)(e))(t))^{a'+b}}\hspace{2mm}_2F_1(a,b;a+a';e-t)dt$
\begin{eqnarray*}
&=&\int_{\Omega\cap
(e-\Omega)}\Delta(e-t)^{a+a'-\frac{n}{r}}\Delta(t)^{b-\frac{n}{r}}\Delta(e-\pi(\pi^{-1}(e+u)(e))(t))^{-a'-b}
\hspace{1mm}_2F_1(a,b;a+a';e-t)dt
\\&=&\frac{\Gamma_\Omega(b)\Gamma_\Omega(a')}{\Gamma_\Omega(a'+b)}\hspace{2mm}_2F_1(a',b;a+a';\pi^{-1}(e+u)(e)).
\end{eqnarray*}
Consequently, the density of $U$ is equal to
\begin{equation}\label{2.9}
f_U(u)=C(a,a',b)\Delta(\pi^{-1}(e+u)(e))^{a'+b}\Delta(u)^{b-\frac{n}{r}}\hspace{2mm}_2F_1(a',b;a+a';\pi^{-1}(e+u)(e))
\textbf{1}_\Omega(u).
\end{equation}
Comparing (\ref{2..10}) and (\ref{2.9}), we conclude that the
densities of $U$ and $V$ are equal, and consequently, their
normalizing constants $C(a,a',b)$ and $C(a',a,b)$ are equal.
\end{proof}

Note that the fact that $C(a,a',b)$ is a symmetric function of
$(a,a')$ means that
\begin{equation}\label{2.123}
\frac{_3F_2(a,a,b;a+b,a+a';e)}{\Gamma_\Omega(a')\Gamma_\Omega(a+b)}=\frac{_3F_2(a',a',b;a'+b,a+a';e)}{\Gamma_\Omega(a)\Gamma_\Omega(a'+b)}.
\end{equation}
 We can deduce another expression of the spherical
Fourier transform of the beta-hypergeometric distribution from the
following more general result.
\begin{proposition}\label{3.1}
\begin{enumerate}
  \item For $t=(t_1,\ldots,t_r)\in \reel^r$ and $s=(s_1,\ldots,s_r)\in \reel^r$, the integral $$I_{a,a',b}(t,s)=\int_{\Omega\cap
(e-\Omega)}\Delta_t(x)\Delta_s(e-x)\mu_{a,a',b}(dx)$$ converges
 if and only if for $i=1,\ldots,r$,
 $$t_i>\frac{i-1}{2}-a,\ \ s_i>\frac{i-1}{2}-b, $$
 and for all $1\leq k\leq r$, $$\sum_{i=1}^k s_i+ka'>1+k(r-k)-k\frac{(r+1)}{2}.$$ In this case, we have
\begin{equation}\label{3.17}
I_{a,a',b}(t,s)=\frac{\Gamma_\Omega(a+b)\Gamma_\Omega(a+t)\Gamma_\Omega(b+s)}{\Gamma_\Omega(a)\Gamma_\Omega(b)\Gamma_\Omega(a+b+t+s)}
\frac{_3F_2(a+t,a,b;a+b+t+s,a+a';e)}{_3F_2(a,a,b;a+b,a+a';e)}.
\end{equation}
  \item We also have under the conditions $\sum_{i=1}^k t_i+ka>1+k(r-k)-k\frac{(r+1)}{2}, \ \textrm{for all}
\ \ 1\leq k\leq r, $
\begin{equation}\label{3.18}
I_{a,a',b}(t,0)=\int_{\Omega\cap
(e-\Omega)}\Delta_t(x)\mu_{a,a',b}(dx)=\frac{_3F_2(a',a'-t,b;a'+b,a+a';e)}{_3F_2(a',a',b;a'+b,a+a';e)}.
\end{equation}
\end{enumerate}
\end{proposition}
\begin{proof} \\
1) For simplicity, we denote $C=C(a,a',b)$. We first calculate the
integral
\begin{eqnarray*} I_{a,a',b}(t,s)&=&\int_{\Omega\cap
(e-\Omega)}\Delta_t(x)\Delta_s(e-x)\mu_{a,a',b}(dx)\\
 &=&C\int_{\Omega\cap (e-\Omega)}\Delta_t(x)\Delta_s(e-x)
\Delta(x)^{a-\frac{n}{r}}\Delta(e-x)^{b-\frac{n}{r}}\hspace{2mm}_2F_1(a,b;a+a';x)dx.
\end{eqnarray*}
As the determinant and the hypergeometric function are $ K$-
invariant, for $ k \in K,$ we have
\begin{eqnarray*}
I_{a,a',b}(t,s)&=&C\int_{\Omega\cap
(e-\Omega)}\Delta_t(x)\Delta_s(e-x)
\Delta(k^{-1}x)^{a-\frac{n}{r}}\Delta(e-k^{-1}x)^{b-\frac{n}{r}}\hspace{2mm}_2F_1(a,b;a+a';k^{-1}x)dx.
\end{eqnarray*}
Setting $y=k^{-1}x$, we get
\begin{eqnarray*}
I_{a,a',b}(t,s)&=&C\int_{\Omega\cap
(e-\Omega)}\Delta_t(ky)\Delta_s(e-ky)
\Delta(y)^{a-\frac{n}{r}}\Delta(e-y)^{b-\frac{n}{r}}\hspace{2mm}_2F_1(a,b;a+a';y)dy\\
&=&C\int_{\Omega\cap (e-\Omega)}\Delta_t(ky)\Delta_s(e-ky)
\Delta(ky)^{a-\frac{n}{r}}\Delta(e-ky)^{b-\frac{n}{r}}
\sum_{m\geq0}\frac{(a)_m(b)_md_m}{(a+a')_m(\frac{n}{r})_m}
\phi_{m}(y)dy.
\end{eqnarray*}
Since all terms are positive we can invert sums and integrals,
whether they converge or not. Hence
\begin{eqnarray*}
I_{a,a',b}(t,s)&=&C\sum_{m\geq0}\frac{(a)_m(b)_md_m}{(a+a')_m(\frac{n}{r})_m}\int_{\Omega\cap
(e-\Omega)}\int_K\Delta_t(ky)\Delta_s(e-ky)\Delta(ky)^{a-\frac{n}{r}}
\Delta(e-ky)^{b-\frac{n}{r}}
\end{eqnarray*}
$\hspace{2.5cm}\Delta_m(ky)dk dy.$\\
It follows that
\begin{equation}\label{3.19}
I_{a,a',b}(t,s)=C\sum_{m\geq0}\frac{(a)_m(b)_m}{(a+a')_m(\frac{n}{r})_m}d_m\int_{\Omega\cap
(e-\Omega)}\Delta_{m+a+t-\frac{n}{r}}(z)\Delta_{s+b-\frac{n}{r}}(e-z)dz.\hspace{0.8cm}
\end{equation}
This last integral converges if and only if $a+t_i>\frac{i-1}{2}$
and $b+s_i>\frac{i-1}{2}$, and under these conditions, it is equal
to
$$C\sum_{m\geq0}\frac{(a)_m(b)_m}{(a+a')_m(\frac{n}{r})_m}d_m\frac{\Gamma_\Omega(m+t+a)\Gamma_\Omega(s+b)}{\Gamma_\Omega(m+t+a+s+b)}.$$
Thus
\begin{eqnarray*}
I_{a,a',b}(t,s)&=&C\hspace{1mm}\sum_{m\geq0}\frac{(a)_m(b)_m}{(a+a')_m(\frac{n}{r})_m}d_m\frac{\Gamma_\Omega(m+t+a)\Gamma_\Omega(s+b)}{\Gamma_\Omega(m+t+a+s+b)}
\\&=&C\sum_{m\geq0}\frac{(a)_m(b)_m(t+a)_md_m}{(a+a')_m(t+a+s+b)_m(\frac{n}{r})_m}
\frac{\Gamma_\Omega(t+a)\Gamma_\Omega(s+b)}{\Gamma_\Omega(t+a+s+b)}.
\end{eqnarray*}
From Proposition \ref{2.2}, this series converges if and only if for
all $1\leq k\leq r$,$$\sum_{i=1}^k
s_i+ka'>1+k(r-k)-k\frac{(r+1)}{2},$$ and under this condition, we
have that
\begin{eqnarray*}
I_{a,a',b}(t,s)&=&\frac{_3F_2(t+a,a,b;t+a+s+b,a+a';e)}{_3F_2(a,a,b;a+b,a+a';e)}
\frac{\Gamma_\Omega(a+b)\Gamma_\Omega(t+a)\Gamma_\Omega(s+b)}{\Gamma_\Omega(a)\Gamma_\Omega(b)\Gamma_\Omega(t+a+s+b)}.
\end{eqnarray*}
2) For this second part, we use Theorem \ref{th1}. Consider two
independent random variables $X\sim \mu_{a,a',b}$ and $W\sim
\beta_{b,a'}^{(2)}$. Then $X'=\pi^{-1}(e+ \pi (X)(W))(e)\sim
\mu_{a',a,b}.$ \\
Since $\pi^{-1}(X')(e-X')= \pi(X)(W)$ and
$E(\Delta_t(W))=\frac{\Gamma_\Omega(b+t)\Gamma_\Omega(a'-t)}{\Gamma_\Omega(b)\Gamma_\Omega(a')}$,
we have
\begin{eqnarray*}
E(\Delta_t(\pi(X)(W)))&=&E(\Delta_t(X))E(\Delta_t(W))
\\&=&E(\Delta_t(\pi^{-1}(X')(e-X')))\\&=&E(\frac{1}{\Delta_t(X')}\Delta_t(e-X')).
\end{eqnarray*}
It follows that
\begin{eqnarray*}
E(\Delta_t(X))&=&\frac{1}{E(\Delta_t(W))}E(\frac{1}{\Delta_t(X')}\Delta_t(e-X'))
\\&=&\frac{1}{E(\Delta_t(W))}E(\Delta_{-t}(X')\Delta_t(e-X')).
\end{eqnarray*}
Now we apply the first part of the proposition by replacing
$(a,a',b,t,s)$ by $(a',a,b,-t,t)$, getting the result for
$\frac{i-1}{2}-b<t_i<a'-\frac{i-1}{2}$ and $\sum_{i=1}^k
t_i+ka>1+k(r-k)-k\frac{(r+1)}{2} \ \textrm{for all} \ \ 1\leq k\leq
r.$ Under these conditions, we obtain that the spherical Fourier
transform of $X$ is
\begin{equation}\label{3.20}
E(\Delta_t(X))=\frac{_3F_2(a',a'-t,b;a'+b,a+a';e)}{_3F_2(a',a',b;a'+b,a+a';e)}.
\end{equation}
Finally, we observe that the right hand side of (\ref{3.20}) is
finite if and only if for all $\ \ 1\leq k\leq r $,  $\sum_{i=1}^k
t_i+ka>1+k(r-k)-k\frac{r+1}{2} \ $ and it is a positive analytic
function of $t$ satisfying this condition. The principle of
maximal analyticity implies that (\ref{3.18}) holds for $t$ such
that for all $1\leq k\leq r $, $\sum_{i=1}^k
t_i+ka>1+k(r-k)-k\frac{(r+1)}{2} .$

\end{proof}
\begin{remark}
\begin{enumerate}
  \item From (\ref{3.17}) and (\ref{3.18}), we obtain two different
expressions of $E(\Delta_t(X))$. Equating these expressions and
recalling (\ref{2.123}), we obtain the following relation concerning
the function $_3F_2$.
\begin{equation}\label{3.22}
\frac{_3F_2(a+t,a,b;t+a+b,a+a';e)}{\Gamma_\Omega(a')\Gamma_\Omega(a+b+t)}=\frac{_3F_2(a',a'-t,b;a'+b,a+a';e)}{\Gamma_\Omega(a+t)
\Gamma_\Omega(a'+b)},
\end{equation}
for $t_i>\frac{i-1}{2}-a$ and $\sum_{i=1}^k
t_i+ka>1+k(r-k)-k\frac{(r+1)}{2} \ \textrm{for all} \ \ 1\leq
k\leq r.$
\item Using the characterization of the beta-hypergeometric distribution
by its spherical Fourier transform given in (\ref{3.20}), we can
easily show the converse of Theorem \ref{th1}, that is if $X$ and $
W$ are two independent random matrices in $\Omega$ such that $W\sim
\beta^{(2)}_{b,a'}$, then $\pi^{-1}(e+ \pi (X)(W))(e)\sim
\mu_{a',a,b}$ implies that $X\sim \mu_{a,a',b}$.
\end{enumerate}
\end{remark}
Next, we give the matrix version of (\ref{Ab4}).
\begin{theorem}
\begin{enumerate}
  \item Let $W\sim \beta_{b,a}^{(2)}$, $W'\sim\beta_{b,a'}^{(2)}$ and
$X $ be three independent random variables, with $X $ valued in $
\Omega$. Then
\begin{equation}\label{2.10}
X\sim \pi^{-1}(e+\pi^{-1}(e+\pi(X)(W'))(W))(e)
\hspace{5mm}if\hspace{2mm}and\hspace{2mm}only\hspace{2mm}if\hspace{5mm}X\sim\mu_{a,a',b}
\end{equation}
  \item If $W\sim \beta_{b,a}^{(2)}$ and $X \in \Omega$ are two independent random variables, then
\begin{equation}\label{2.11}
X\sim \pi^{-1}(e+\pi(X)(W))(e)
\hspace{5mm}if\hspace{2mm}and\hspace{2mm}only\hspace{2mm}if\hspace{5mm}X\sim\mu_{a,a,b}
\end{equation}
\item Let $(W_n)_{n\geq 1}$ and $(W'_n)_{n\geq 1}$ be two
independent sequences of random variables with respective
distributions $\beta_{b,a}^{(2)}$ and $\beta_{b,a'}^{(2)}$. Then
$\mu_{a,a',b}$ is the distribution of the random continued
fraction
\begin{equation}\label{2.12}
\frac{e}{e+\frac{W_1}{e+\frac{W'_1}{e+\frac{W_2}{e+\frac{W'_2}{e+\ldots}}}}}.
\end{equation}
\end{enumerate}
\end{theorem}
\begin{proof}
We adapt the method of proof used in the real case by Asci, Letac
and Piccioni \cite{6}  to the matrix case.
\begin{enumerate}
\item Observe first the series (\ref{2.12}) converges almost surely, because
the series $\sum_n(W^{-1}_n+W'^{-1}_n )$ diverges
almost surely. Consider the sequence $(F_n)_{n=1}^\infty$ of random
mappings from $\Omega\cap (e-\Omega)$ into itself defined by
$$F_n(z)=\pi^{-1}(e+\pi^{-1}(e+\pi(z)(W'_n))(W_n))(e).$$
Since $F_1\circ \ldots\circ F_n(z)$ has almost surely a limit $X$,
then the distribution of $X$ is a stationary distribution of the
Markov chain $w_n=F_n\circ \ldots\circ F_1(z)$ which is unique.
According to Theorem (\ref{th1}), we have that $\mu_{a,a',b}$ is a
stationary distribution of the Markov chain $(w_n)_{n=0}^\infty$. It
follows that $X\sim \mu_{a,a',b}$.
\item We use the reasoning above with the
random mappings $ G_{n}(z)=\pi^{-1}(e+\pi(z)(W_{n}))(e)$. \item The
proof of this part is similar to the first part.
\end{enumerate}
\end{proof}\\
In the following theorem, we establish the identifiability of the
beta-hypergeometric distribution on symmetric matrices.
\begin{theorem} \label{1}
Let $(a,a',b)$ and $(a_1,a_1',b_1)$ in $
(]\frac{r-1}{2},\infty[)^3$.
$$ \textrm{If}\ \ \mu_{a,a',b}=\mu_{a_1,a_1',b_1},\ \
\textrm{then} \ \ (a,a',b)=(a_1,a_1',b_1).$$
\end{theorem}
\begin{proof}
 For the sake of simplification, we denote  $C=C(a,a',b)$ and $C_1=C(a_1,a_1',b_1)$.\\
For $x \in \Omega\cap (e-\Omega)$, we have that
\begin{eqnarray*}
\lim_{x\rightarrow
0}\Delta(e-x)^{b-\frac{n}{r}}\hspace{2mm}_2F_1(a,b;a+a';x)
&=&\lim_{x\rightarrow 0}\Delta(e-x)^{b-\frac{n}{r}}\sum_{m\geq
0}\frac{(a)_m(b)_m}{(a+a')_m}\frac{d_m}{(\frac{n}{r})_m}\phi_m(x)
\\&=&\lim_{x\rightarrow
0}\Delta(e-x)^{b-\frac{n}{r}}d_{(0,\ldots,0)}  \\&+&
\lim_{x\rightarrow 0}\Delta(e-x)^{b-\frac{n}{r}}\sum_{m\geq 0
;m\neq
0}\frac{(a)_m(b)_m}{(a+a')_m}\frac{d_m}{(\frac{n}{r})_m}\phi_m(x)\\&=&\lim_{x\rightarrow
0}\Delta(e-x)^{b-\frac{n}{r}}d_{(0,\ldots,0)}\\&+&\lim_{x\rightarrow
0}\sum_{m\geq 0 ;m\neq
0}\Delta(e-x)^{b-\frac{n}{r}}\frac{(a)_m(b)_m}{(a+a')_m}\frac{d_m}{(\frac{n}{r})_m}\phi_m(x).
\end{eqnarray*}
Since $d_{(0,\ldots,0)}=1$, and $$\lim_{x\rightarrow 0}\sum_{m\geq 0
;m\neq
0}\Delta(e-x)^{b-\frac{n}{r}}\frac{(a)_m(b)_m}{(a+a')_m}\frac{d_m}{(\frac{n}{r})_m}\phi_m(x)=0,$$
we conclude that $$\lim_{x\rightarrow
0}\Delta(e-x)^{b-\frac{n}{r}}\hspace{2mm}_2F_1(a,b;a+a';x)=1.$$
Thus, when $x$ is close to $0$, the densities of $\mu_{a,a',b}$ and
$\mu_{a_1,a_1',b_1}$ are respectively equivalent to $C
\Delta(x)^{a-\frac{n}{r}}$ and $ C_1 \Delta(x)^{a_1-\frac{n}{r}}$.
Since $\mu_{a,a',b}=\mu_{a_1,a_1',b_1}$, we get $C
\Delta(x)^{a-\frac{n}{r}}=C_1 \Delta(x)^{a_1-\frac{n}{r}}$. Hence
$a=a_1$, and it follows that for all  $x $ in $ \Omega\cap
(e-\Omega)$,
$$\Delta(e-x)^{b-\frac{n}{r}}\hspace{2mm}_2F_1(a,b;a+a';x)=\Delta(e-x)^{b_1-\frac{n}{r}}\hspace{2mm}_2F_1(a,b_1;a+a_1';x).$$
 Using Proposition $XV.3.4$ page $330$ in \cite{FK}, we
can write for $x $ in $ \Omega\cap (e-\Omega)$,
$$_2F_1(a,b;a+a';x)=\Delta(e-x)^{-b}\hspace{2mm}_2F_1(a',b;a+a';-x(e-x)^{-1})$$
and
$$_2F_1(a,b_1;a+a_1';x)=\Delta(e-x)^{-b_1}\hspace{2mm}_2F_1(a_1',b_1;a+a_1';-x(e-x)^{-1}).$$
Therefore
$$_2F_1(a',b;a+a';z)=\hspace{2mm}_2F_1(a_1',b_1;a+a_1';z),$$
or equivalently
$$\sum_{m\geq 0}\frac{(a')_m(b)_m}{(c)_m}\frac{d_m}{(\frac{n}{r})_m}\phi_m(z)
=\sum_{m\geq
0}\frac{(a_1')_m(b_{1})_m}{(c_1)_m}\frac{d_m}{(\frac{n}{r})_m}\phi_m(z)$$
where $c=a+a'$ and $c_1=a+a_1'$.\\
Since $\phi_m(z)$ is a polynomial in $z$ with degree equal to
$|m|=m_1+\ldots+m_r$, this implies that
$$\frac{(a')_m(b)_m}{(c)_m}\frac{d_m}{(\frac{n}{r})_m}= \frac{(a_1')_m(b_{1})_m}{(c_1)_m}\frac{d_m}{(\frac{n}{r})_m},$$for each $m\geq 0$.\\
For $m=(1,0,\ldots,0)$,
we obtain
 $$\frac{a'b}{c}=\frac{a_1'b_1}{c_1}.$$
For $m=(1,1,0,\ldots,0)$, we obtain
 $$\frac{(a'-\frac{1}{2})(b-\frac{1}{2})}{(c-\frac{1}{2})}=\frac{(a_1'-\frac{1}{2})(b_1-\frac{1}{2})}
 {(c_1-\frac{1}{2})}.$$
Finally, for $m=(1,1,1,0,\ldots,0)$, we obtain
$$\frac{(a'-1)(b-1)}{(c-1)}=\frac{(a_1'-1)(b_1-1)}
 {(c_1-1)}.$$
Let
$$\lambda_0=\frac{a'b}{c},\ \ \lambda_1=\frac{(a'-\frac{1}{2})(b-\frac{1}{2})}{(c-\frac{1}{2})} \ \
\textrm{and} \ \ \lambda_2=\frac{(a'-1)(b-1)}{(c-1)}.$$ By taking
suitable linear combination, we get
$$a'b=c\lambda_0, \
a'+b=2c\lambda_0-2(c-\frac{1}{2})\lambda_1+\frac{1}{2}, \
c(\lambda_2+\lambda_0-2\lambda_1)+\lambda_1-\lambda_2-\frac{1}{2}=0.$$
From this, $c$ can be uniquely determined, we get $c=c_1$ then
$a'=a_1'$, and $b=b_1$.
\end{proof}
\begin{theorem}
Let $X $ be a beta-hypergeometric random matrix, $X \sim
\mu_{a,a',b}$. Then \\$(e-X) \sim \mu_{a_1,a_1',b_1}$ if and only
if $a_1=a'$, $a=a_1'$ and $b_1=b=a+a'=a_1+a'_1$.
\end{theorem}
\begin{proof}
 $(\Leftarrow)$ This way is obvious.\\
$(\Rightarrow)$ Suppose that $X\sim \mu_{a,a',b}$ and $ e-X\sim
\mu_{a_1,a_1',b_1}$. Since the beta-hypergeometric distribution  is
$K$-invariant, it is characterized by its spherical Fourier
transform.
\begin{eqnarray*}
E(\Delta_t(e-X))&=&C(a,a',b)\int_{\Omega\cap
(e-\Omega)}\Delta_t(e-x)\Delta(x)^{a-\frac{n}{r}}\Delta(e-x)^{b-\frac{n}{r}}\hspace{2mm}
_2F_1(a,b;a+a';x)dx\\&=&C(a,a',b)\int_{\Omega\cap
(e-\Omega)}\Delta_{t+b-\frac{n}{r}}(e-x)\Delta(x)^{a-\frac{n}{r}}\hspace{2mm}
_2F_1(a,b;a+a';x)dx.\\
&=&C(a,a',b)\int_{\Omega\cap
(e-\Omega)}\Delta_{t+b-\frac{n}{r}}(e-x)\Delta(x)^{a-\frac{n}{r}}\hspace{2mm}
_2F_1(a,b;a+a';k^{-1}x)dx,
\end{eqnarray*}
where the last equality is due to the fact that the hypergeometric
function is $K$-invariant.\\
Setting $y=k^{-1}x$, we obtain that
\begin{eqnarray*}
E(\Delta_t(e-X))&=&C(a,a',b)\int_{\Omega\cap (e-\Omega)}\Delta_{t+b-\frac{n}{r}}(e-ky)\Delta(ky)^{a-\frac{n}{r}}\hspace{2mm}
_2F_1(a,b;a+a';y)dy\\&=&C(a,a',b)\sum_{m\geq 0}\frac{(a)_m(b)_m}{(a+a')_m}\frac{d_m}{(\frac{n}{r})_m}\int_{\Omega\cap (e-\Omega)}\int_K\Delta_{t+b-\frac{n}{r}}(e-ky)\Delta_{m+a-\frac{n}{r}}(ky)dkdy\\&=&C(a,a',b)\sum_{m\geq 0}\frac{(a)_m(b)_m}{(a+a')_m}\frac{d_m}{(\frac{n}{r})_m}\int_{\Omega\cap (e-\Omega)}\Delta_{t+b-\frac{n}{r}}(e-y)\Delta_{m+a-\frac{n}{r}}(y)dy.
\end{eqnarray*}
The last integral converges when $t_i+b>\frac{i-1}{2}\ \
 \textrm{for all } 1\leq i\leq r.$ Under this condition we can write that
\begin{eqnarray*}
E(\Delta_t(e-X))&=&C(a,a',b)\sum_{m\geq
0}\frac{(a)_m(b)_m}{(a+a')_m}\frac{d_m}{(\frac{n}{r})_m}\frac{\Gamma_\Omega(m+a)\Gamma_\Omega(t+b)}{\Gamma_\Omega(m+a+t+b)}\\&=&
\sum_{m\geq
0}\frac{(a)_m(b)_m(a)_m}{(a+a')_m(a+t+b)_m}\frac{d_m}{(\frac{n}{r})_m}\frac{\Gamma_\Omega(a+b)\Gamma_\Omega(t+b)}{\Gamma_\Omega(b)\Gamma_\Omega(a+t+b)\hspace{1mm}
_3F_2(a,a,b;a+b,a+a';e)}\\&=&\frac{\Gamma_\Omega(a+b)\Gamma_\Omega(t+b)}{\Gamma_\Omega(b)\Gamma_\Omega(a+t+b)}\frac{\hspace{1mm}_3F_2(a,a,b;a+a',a+b+t;e)}
{\hspace{1mm}_3F_2(a,a,b;a+b,a+a';e)},
\end{eqnarray*}
which is defined for $t$ such that
$\sum_{i=1}^kt_i+ka'>1+k(r-k)-\frac{k(r+1)}{2}\ \ \textrm{for all}\
\ 1\leq k\leq r.$ Also since $e-X\sim \mu_{a_1,a_1',b_1}$, we use
(\ref{delta}) for $t$ such that for $1\leq i\leq r$,
$t_i+a_{1}>\frac{i-1}{2}$, to obtain that
$$\frac{\Gamma_\Omega(a+b)\Gamma_\Omega(t+b)}{\Gamma_\Omega(b)\Gamma_\Omega(a+t+b)}\frac{\hspace{1mm}_3F_2(a,a,b;a+a',a+b+t;e)}
{\hspace{1mm}_3F_2(a,a,b;a+b,a+a';e)}$$$$=\frac{\Gamma_\Omega(a_1+b_1)\Gamma_\Omega(t+a_1)}{\Gamma_\Omega(a_1)\Gamma_\Omega(a_1+t+b_1)}
\frac{\hspace{1mm}_3F_2(a_1,b_1,a_1+t;a_1+a_1',a_1+b_1+t;e)}
{\hspace{1mm}_3F_2(a_1,a_1,b_1;a_1+b_1,a_1+a_1';e)}.$$ The last
equality is equivalent to
$$\frac{\hspace{1mm}_3F_2(a_1,b_1,a_1+t;a_1+a_1',a_1+b_1+t;e)}{\hspace{1mm}_3F_2(a,a,b;a+a',a+b+t;e)}$$$$=\frac{\Gamma_\Omega(a+b)\Gamma_\Omega(a_1)}
{\Gamma_\Omega(b)\Gamma_\Omega(a_1+b_1)}\frac{\hspace{1mm}_3F_2(a_1,a_1,b_1;a_1+b_1,a_1+a_1';e)}{\hspace{1mm}_3F_2(a,a,b;a+b,a+a';e)}\frac
{\Gamma_\Omega(t+b)\Gamma_\Omega(a_1+t+b_1)}
{\Gamma_\Omega(a+t+b)\Gamma_\Omega(t+a_1)}. $$ Hence the function
$t\mapsto\frac{_3F_2(a_{1},b_{1},a_{1}+t;a_{1}+a_{1}',t+a_{1}+b_{1};e)}{_3F_2(a,a,b;a+a',t+a+b;e)}$
is expressed in terms of gamma functions. This happens if and only
if $b_{1}=a_{1}+a_{1}'$ and $b=a+a'$. In this case we have that
$e-X\sim \beta_{a_1,a'_1}^{(1)}$ and $X\sim \beta_{a,a'}^{(1)}$.
However when $X\sim \beta_{a,a'}^{(1)}$, we have that $e-X \sim
\beta_{a',a}^{(1)}$. Hence $a_1=a'$ and $a'_1=a$.
\end{proof}
\begin{prop}
The following convergences in law hold: \begin{eqnarray*} 1. &
&\lim_{a\rightarrow 0}\mu_{a,a',b}=\delta_0\  \textrm{and} \
\lim_{a\rightarrow
  \infty}\mu_{a,a',b}=\delta_e,\\
   2. & &\lim_{a'\rightarrow 0}\mu_{a,a',b}=\delta_e \ and \  \lim_{a'\rightarrow
  \infty}\mu_{a,a',b}=\beta^{(1)}_{a,b},\\
   3.& & \lim_{b\rightarrow 0}\mu_{a,a',b}=\delta_e, \lim_{b\rightarrow
  \infty}\mu_{a,a',b}=\delta_0 \ if\  a-\frac{r-1}{2}\leq a' \ and \  \lim_{b\rightarrow
  \infty}\mu_{a,a',b}=\beta^{(1)}_{a-a',a'} \ if \ a'<a-\frac{r-1}{2}.
  \end{eqnarray*}
\end{prop}
\begin{proof}
For the proof, we will use the spherical
Fourier transform.\\
1) For the first part, we need to show that, for  $ t\in \reel^{r}$
such that $\sum_{i=1}^r t_i> 1+k(r-k)-\frac{k(r+1)}{2}$,  $ 1 \leq
k\leq r,$ $$\lim_{a\rightarrow0}
\frac{_3F_2(a',a'-t,b;a'+b,a+a';e)}{_3F_2(a',a',b;a'+b,a+a';e)}=0.$$
In fact, as $_3F_2(a',a'-t,b;a'+b,a+a';e)=\sum_{m\geq
0}\frac{(a'-t)_m(b)_m d_m}{(a'+b)_m(\frac{n}{r})_m}\times
\frac{(a')_m}{(a+a')_m},$ then using the fact that $a\mapsto
\frac{(a')_m}{(a+a')_m}$ is a decreasing function of $a$ on
$(\frac{r-1}{2},\infty)$ and that
$$\sum_{m\geq 0}|\frac{(a'-t)_m(b)_m
d_m}{(a'+b)_m(\frac{n}{r})_m}|<\infty,$$ the monotone convergence
theorem enables us to invert sum and limit to get
$$ \lim_{a\rightarrow 0} {_3F_2(a',a'-t,b;a'+b,a+a';e)}=\hspace{1mm} _2F_1(a'-t,b;a'+b;e).$$
Similarly, we show that$$ \lim_{a\rightarrow 0}{
_3F_2(a',a',b;a'+b,a+a';e)}= \infty.$$ Therefore
$$\lim_{a\rightarrow 0}\mu_{a,a',b}=\delta_0.$$ On the other hand,
when $a\rightarrow \infty$, all the terms in the numerator and in
the denominator go to zero except the one corresponding to $
m=(0,...,0)$ which is equal to 1. Thus spherical
Fourier transform tends to 1, which implies that
$$\lim_{a\rightarrow
  \infty}\mu_{a,a',b}=\delta_e.$$
2) We will use the result established in Theorem \ref{th1}, that is
if $ X \sim \mu_{a',a,b}$ is independent of $ W \sim \beta^{(2)}_{b,
a}$, then $ \pi^{-1}(e+\pi(X)(W))(e)  \sim \mu_{a, a',b}.$ Thus
according to the point 1) established above, we have that
$$\textrm{if}\ \ \lim_{a'\rightarrow 0}\mu_{a',a,b}=\delta_0,\ \
\textrm{then}\ \ \lim_{a'\rightarrow 0}\mu_{a,a',b}=\delta_e.$$\\
With the same reasoning, we see that $$\lim_{a'\rightarrow
  \infty}\mu_{a,a',b}=\beta^{(1)}_{a,b}.$$
3) Similarly, except the term corresponding to $ m=(0,...,0)$, which
is equal to 1, all the other terms in the numerator and in the
denominator go to zero when $b$ tends to zero. Thus the spherical
Fourier transform tends to 1, which implies that
$$\lim_{b\rightarrow 0}\mu_{a,a',b}=\delta_e.$$
 When $b\rightarrow\infty$, we have, for $ t \in \reel^{r}$ such that $ \sum_{i=1}^{k} t_{i} + k(a-a') > 1+ k(r-k)-k
  \frac{(r+1)}{2},\ \  1\leq k\leq r,$ \ \
$$ \lim_{b\rightarrow
 \infty}\ \ _3F_2(a',a'-t,b;a'+b,a+a';e)=\lim_{b\rightarrow
  \infty}\sum_{m\geq 0}\frac{(a'-t)_m(a')_m d_m}{(a+a')_m(\frac{n}{r})_m}\times
\frac{(b)_m}{(a'+b)_m}.$$ Here also, we can invert the sum and the
limit to obtain that
$$\lim_{b\rightarrow
  \infty}\hspace{2mm} _3F_2(a',a'-t,b;a'+b,a+a';e)= \hspace{2mm}_2F_1(a',a'-t;a+a';e).$$
  Similarly, $$\lim_{b\rightarrow
  \infty}\hspace{2mm} _3F_2(a',a',b;a'+b,a+a';e)=\hspace{2mm}_2F_1(a',a';a+a';e).$$
In the case where $ a' < a-\frac{r-1}{2},$ the spherical Fourier
transform of $\mu_{a, a', b}$ tends to
$$\frac{_2F_1(a',a'-t;a+a';e)}{_2F_1(a',a';a+a';e)}=
 \frac{\Gamma_\Omega(a)\Gamma_\Omega(a-a'+t)}{\Gamma_\Omega(a-a')\Gamma_\Omega(a+t)},$$ which is the
spherical Fourier transform of $\beta^{(1)}_{a-a',a'}.$
\\ Otherwise the spherical
Fourier transform of $\mu_{a, a', b}$ tends to 0, in which case
$$\lim_{b\rightarrow
  \infty}\mu_{a,a',b}=\delta_0.$$
\end{proof}

\end{document}